\documentclass[12pt]{amsart}

\usepackage{amssymb}

\usepackage{enumitem}

\makeatletter
\@namedef{subjclassname@2010}{%
  \textup{2010} Mathematics Subject Classification}
\makeatother

\usepackage[T1]{fontenc}
\newtheorem{thm}{Theorem}
\newtheorem*{que*}{Question}
\newtheorem{lem}{Lemma}

\numberwithin{equation}{section}

\newcommand{\Z}{\mathbb Z}
\newcommand{\N}{\mathbb N}
\newcommand{\eps}{\varepsilon}

\newcommand{\dd}{\mathrm{d}}
\newcommand{\push}[2]{{#2}_*(#1)}
\newcommand{\pushh}[1]{{#1}_*}
\newcommand{\acts}{\curvearrowright}
\newcommand{\folner}{F\o{}lner }

\newcommand{\symdiff}{\triangle}
\DeclareMathOperator{\graph}{graph}

\DeclareMathOperator{\id}{id}

\renewcommand{\phi}{\varphi}

\begin{document}

%%%%% To ease editing, for IMPAN journals add:

\baselineskip=17pt

%%%%%%%%%%%%%%%%

\title{On lifting invariant probability measures}

\author{Tomasz Cie\'sla}
\address{Department of Mathematics and Statistics, McGill University, 805, Sherbrooke Street West Montreal, Quebec, Canada H3A 2K6}
\email{tomasz.ciesla@mail.mcgill.ca}

\date{}

\begin{abstract}
In this note we study when an invariant probability measure lifts to an invariant measure. Consider a standard Borel space $X$, a Borel probability measure $\mu$ on $X$, a Borel map $T \colon X \to X$ preserving $\mu$, a Polish space $Y$, a continuous map $S\colon Y \to Y$, and a Borel surjection $p \colon Y \to X$ with $p\circ S = T \circ p$. We prove that if fibers of $p$ are compact then $\mu$ lifts to an $S$-invariant measure on $Y$. 
\end{abstract}



\maketitle

\section{Introduction}

In this note we address the following question asked by Feliks Przytycki:

\begin{que*}
Let $X$ be a compact metric space and $Y$ a Polish space. Let $T\colon X \to X$, $S\colon Y \to Y$ be continuous maps. Let $p\colon Y \to X$ be a Borel surjection with $p \circ S = T \circ p$. Let $\mu$ be a $T$-invariant Borel probability measure on $X$. When does $\mu$ lift to an $S$-invariant Borel probability measure on $Y$?
\end{que*}

The answer is affirmative under the assumption that fibers of $p$ are finite and the sets $\{x \in X \colon |p^{-1}(x)|=n\}$ are $T$-invariant (for instance, this holds if $S$ and $T$ are homeomorphisms). A special case of this ($|p^{-1}(x)|\le 2$ for all $x\in X$) appeared in the proof of \cite[Corollary 10.2]{przytycki}. An obvious modification of Przytycki's argument shows that one can lift $\mu$ to an $S$-invariant measure $\nu$ where $\nu$ is defined by
$$\nu(A)=\int_X \frac{|A \cap p^{-1}(x)|}{|p^{-1}(x)|} \dd \mu(x).$$

It is also known that if $Y$ is compact and $p$ is continuous then $\mu$ lifts to an $S$-invariant measure $\nu$. Note that $p$ induces the push-forward map $\pushh{p} \colon P(Y) \to P(X)$ (here, $P(Y)$ and $P(X)$ denote the spaces of all Borel probability measures on $Y$ and $X$, respectively) which is a continuous surjection, so the preimage of $\mu$ is a non-empty compact subset $K$ of $P(Y)$. Clearly, $K$ is convex. Since $\mu$ is $T$-invariant and $p\circ S = T \circ p$, we obtain $\pushh{S}(K)\subset K$. Hence by Schauder's fixed-point theorem there exists $\nu \in K$ with $\nu=\push{\nu}{S}$. This means: $\nu$ is a lift of $\mu$ which is $S$-invariant.

On the other hand, if the assumption on compactness of fibers of $p$ is dropped then it may happen that $\mu$ does not lift to an $S$-invariant measure even if $Y$ is compact, $T$ is the identity map and $S$ is a homeomorphism. For instance, let $X=\{0, 1\}$ and $Y = \Z \cup \{\infty\}$ be the one-point compactification of the countable discrete space $\Z$. Let $T=\id_X$, $S(n)=n+1$ for $n\in \Z$, $S(\infty)=\infty$, $p(n)=0$ for $n\in \Z$, $p(\infty)=1$, and $\mu=\frac 12\delta_0+\frac 12\delta_1$. Suppose that $\nu$ is an $S$-invariant measure on $Y$. By $S$-invariance, $\nu(\{n\})=\nu(\{0\})$ for all $n\in \Z$. If $\nu(\{0\})=0$ then $\nu(\Z)=\sum_{n\in \Z} \nu(\{n\})=0$ and if $\nu(\{0\})>0$ then $\nu(\Z)=\sum_{n\in \Z} \nu(\{n\})=\infty$. In both cases $\nu(\Z)\neq \frac 12$, hence $\mu$ does not lift to an $S$-invariant measure.

We shall work in a more general context. We drop the assumption on compactness of $X$ and continuity of $T$. The following result generalizes both special cases discussed above.

\begin{thm} \label{thm:main1}
Let $X$ be a standard Borel space with a Borel probability measure $\mu$ and let $T:X\to X$ be a $\mu$-measurable map preserving $\mu$. Let $Y$ be a Polish space and let $S \colon Y \to Y$ be a continuous map. Let $p \colon Y \to X$ be a Borel map such that $p \circ S=T\circ p$ and $\mu(p(Y))=1$. Suppose that for $\mu$-a.a. $x\in X$ the set $p^{-1}(x)$ is compact. Then there exists a Borel probability measure $\nu$ on $Y$ which is $S$-invariant and $\push{\nu}{p}=\mu$.
\end{thm}

One can prove even more general result: instead of single maps $S$ and $T$ one can work with a left amenable semigroup $\Gamma$ (for instance, an abelian semigroup) acting on $Y$ by continuous maps and acting on $X$ by measure-preserving maps so that the actions of $\Gamma$ on $Y$ and $X$ commute with $p$.

\begin{thm} \label{thm:main2}
Let $X$ be a standard Borel space with a Borel probability measure $\mu$. Let $Y$ be a Polish space. Let $p \colon Y \to X$ be a Borel map with $\mu(p(Y))=1$ and such that the set $p^{-1}(x)$ is compact for $\mu$-a.a. $x\in X$. Let $\Gamma$ be a left amenable semigroup. Consider actions $\Gamma \acts Y$, $\Gamma \acts X$ so that:
\begin{itemize}
\item $\Gamma$ acts on $Y$ by continuous maps, i.e. for all $\gamma \in \Gamma$ the map $S_\gamma \colon Y \to Y$, $S_\gamma(y)=\gamma y$ is continuous,
\item $\mu$ is $\Gamma$-invariant, i.e. for all $\gamma \in \Gamma$ the map $T_\gamma \colon X \to X$, $T_\gamma(x)=\gamma x$ preserves $\mu$,
\item The actions of $\Gamma$ on $Y$ and $X$ commute with $p$, i.e. $p \circ S_\gamma = T_\gamma \circ p$ for all $\gamma \in \Gamma$.
\end{itemize}
Then there exists a $\Gamma$-invariant Borel probability measure $\nu$ on $Y$ such that $\push{\nu}{p}=\mu$.
\end{thm}

Clearly, Theorem \ref{thm:main1} is a special case of Theorem \ref{thm:main2}; to see this just take $\Gamma=(\N,+)$ with actions on $X$ and $Y$ given by $\N \times X \ni (n,x) \mapsto T^n x \in X$ and $\N \times Y \ni (n,y) \mapsto S^n y \in Y$, respectively. Therefore it is enough to prove Theorem \ref{thm:main2}. Nevertheless, we provide a proof of Theorem \ref{thm:main1} which avoids using tools from theory of amenable semigroups.

\section{Preliminaries}

In this section we recall some definitions and useful facts.

A standard Borel space is an uncountable set $X$ with a $\sigma$-algebra $\Sigma$ of subsets of $X$ such that there exists a Polish (i.e. separable, completely metrizable) topology $\tau$ on $X$ whose Borel $\sigma$-algebra is $\Sigma$.

Given a topological space $Y$ we denote by $K(Y)$ the collection of all compact subsets of $Y$. The set $K(Y)$ can be endowed with Vietoris topology, i.e. the topology generated by sets 
$$\{K \in K(Y) \colon K \cap U \neq \emptyset\} \quad \text{ and } \quad \{K \in K(Y) \colon K \subset U \}$$
where $U\subset Y$ is open. If $Y$ is Polish, compact, then $K(Y)$ is Polish, compact, respectively. 

For a Polish space $Y$ we denote by $P(Y)$ the set of all Borel probability measures on $Y$ endowed with the weak$^*$ topology, i.e. the topology 
generated by sets of the form
$$\left\{\sigma \in P(Y) \colon \left|\int_Y f \dd \sigma - \int_Y f \dd \sigma_0 \right| < \eps \right\}$$
where $\sigma_0\in P(Y)$, $f \colon Y \to \mathbb R$ is continuous and bounded, and $\eps>0$. Traditionally, a somewhat erroneous terminology is in use: a sequence of measures convergent in the weak$^*$ topology is sometimes said to converge weakly. If $Y$ is a compact metric space then $P(Y)$ is a compact metric space.

A semigroup $\Gamma$ is called left amenable if there exists a left invariant mean for $\Gamma$. For a more detailed definition of left amenable semigroups we refer the reader to \cite[0.18]{paterson}.

\section{Proof of Theorems \ref{thm:main1} and \ref{thm:main2}}

We start with the following key lemma.

\begin{lem}\label{lem:1}
Let $X$ be a standard Borel space with a Borel probability measure $\mu$. Let $Y$ be a Polish space. Let $p \colon Y \to X$ be a Borel map such that $\mu(p(Y))=1$. Let $M\subset P(Y)$ be the set of all measures $\sigma$ with $\push{\sigma}{p}=\mu$.
If for $\mu$-a.a. $x\in X$ the set $p^{-1}(x)$ is compact then $M$ is a non-empty convex compact subset of $P(Y)$.
\end{lem}

\begin{proof}
Suppose additionally that $Y$ is compact. The general case will be considered later.

First of all, the set $M$ is non-empty. For instance, by \cite[18.3]{kechris} there exists a $\mu$-measurable function $u \colon p(Y) \to Y$ with $u(x)\in p^{-1}(x)$ for all $x\in p(Y)$. Define a measure $\sigma \in P(Y)$ by $\sigma(B) = \int_{p(Y)} \delta_{u(x)}(B) \dd \mu(x)$. Then $\sigma \in M$. Secondly, it is clear that $M$ is convex. It remains to prove that $M$ is compact. Let $\nu_1, \nu_2, \nu_3, \ldots$ be a sequence of elements of $M$ convergent to some $\nu \in P(Y)$. We shall prove that $\nu \in M$, i.e. that $\push{\nu}{p}=\mu$.

\medskip

\noindent{\it Claim 1.} 
Let $A \subset X$ be a Borel set. Then $\push{\nu}{p}(A)\ge \mu(A)$.

\begin{proof}[Proof of Claim 1] 
This is trivial if $\mu(A)=0$, so let us assume that $\mu(A)>0$. Fix $\eps>0$. Endow $X$ with a Polish topology giving $X$ its Borel structure. Let $X' \subset X$ be a Borel set of full measure such that for all $x\in X'$ the set $p^{-1}(x)$ is compact.

Let $f \colon X' \to K(Y)$ be given by $f(x)=p^{-1}(x)$. We shall prove that $f$ is Borel. Recall that the Borel structure of $K(Y)$ is generated by sets of the form $B=\{K\in K(Y) \colon K \cap U \neq \emptyset\}$ where $U \subset Y$ is open (see \cite[12.C]{kechris}). Therefore it is enough to prove that the set $f^{-1}(B)$ is Borel whenever $B$ is of the aforementioned form. Note that 
\begin{align*}
f^{-1}(B) & = \{x\in X' \colon f(x) \in B\} = \{x \in X' \colon f(x) \cap U \neq \emptyset\} \\
&=\{x\in X' \colon \exists y\in U \ p(y)=x \} = \pi_X(\graph(p) \cap (U \times X')),
\end{align*}
which is Borel by \cite[28.8]{kechris}. Hence $f$ is Borel.

By Lusin's Theorem there exists a non-empty compact subset $K \subset A\cap X'$ such that $\mu(K)>\mu(A)-\eps$ and the function $f|_K \colon K \to K(Y)$ is continuous.
Then the set $\{f(x) \colon x \in K\}$ is compact in $K(Y)$, as it is a continuous image of a compact set. By \cite[4.29]{kechris}, the set $f(K)=\bigcup\{f(x) \colon x \in K\}=p^{-1}(K)$ is a compact subset of $Y$. 

Since $\nu_n$ converges to $\nu$ weakly and $p^{-1}(K)$ is compact, we have by Portmanteau lemma 
$$\push{\nu}{p}(K) = \nu(p^{-1}(K)) \ge \limsup_{n\to\infty} \nu_n(p^{-1}(K)) = \limsup_{n\to\infty} \mu(K) = \mu(K).$$
It follows that $\push{\nu}{p}(A) \ge \push{\nu}{p}(K) \ge \mu(K) \ge \mu(A)-\eps$. Since $\eps>0$ can be chosen arbitrarily, the claim follows.
\end{proof}

\noindent {\it Claim 2.} 
Let $A \subset X$ be a Borel set. Then $\push{\nu}{p}(A) \le \mu(A)$.

\begin{proof}[Proof of Claim 2]
Claim 1 for the set $X \setminus A$ gives $\push{\nu}{p}(X \setminus A)\ge \mu(X\setminus A)$. This can be rewritten as $1-\push{\nu}{p}(A) \ge 1-\mu(A)$, hence $\push{\nu}{p}(A) \le \mu(A)$.
\end{proof}

Claims 1 and 2 imply that $\push{\nu}{p}(A)=\mu(A)$ for all Borel sets $A \subset X$. Therefore $\push{\nu}{p}=\mu$, which proves that $M$ is closed in $P(Y)$ and hence compact. This finishes the proof in the case when $Y$ is compact.

It remains to consider the case when $Y$ is non-compact. Recall that any Polish space embeds homeomorphically into the Hilbert cube $[0,1]^\N$ as a $G_\delta$ subset. Write $Y'=[0,1]^\N$ for brevity and view $Y$ as a subspace of $Y'$. Let $\Sigma$ be the Borel $\sigma$-algebra of $X$. Let $X'=X \cup \{*\}$. Let $\Sigma'=\Sigma \cup \{A \cup \{*\} \colon A \in \Sigma\}$. Then $\Sigma'$ gives $X'$ a structure of standard Borel space. Let $\mu'$ be a Borel probability measure on $X'$ given by $\mu'(B)=\mu(B\cap X)$ for any $B \in \Sigma'$. Let $p'\colon Y' \to X'$ be given by $p'(y)=p(y)$ if $y\in Y$ and $p'(y)=*$ otherwise. Note that $p'$ is Borel. Let $M'\subset P(Y')$ be the set of all measures $\sigma'$ with $\push{\sigma'}{p}=\mu'$. Then $X'$, $Y'$, $\mu'$, $p'$, and $M'$ satisfy the hypotheses of the lemma and in addition $Y'$ is compact, so $M'$ is a non-empty convex subset of $P(Y')$. It is clear that the map $M \ni \sigma \mapsto \sigma' \in P(Y')$ given by $\sigma'(B)=\sigma(B \cap Y)$ maps $M$ onto $M'$ homeomorphically. Therefore $M$ is a non-empty compact subset of $P(Y)$, which obviously is convex as well.
\end{proof}

We prove Theorem \ref{thm:main1} using the averaging trick. 

\begin{proof}[Proof of Theorem \ref{thm:main1}]
Let $M\subset P(Y)$ be the set of all measures $\sigma$ with $\push{\sigma}{p}=\mu$. By Lemma \ref{lem:1}, $M$ is non-empty, convex and compact.

Pick an arbitrary $\sigma \in M$. For all positive integers $n$ define 
$$\nu_n = \frac 1n \sum_{i=0}^{n-1} \push{\sigma}{(S^i)}.$$ 
Note that for all $i$
$$\push{\push{\sigma}{(S^i)}}{p}  = \push{\sigma}{(p\circ S^i)} = \push{\sigma}{(T^i \circ p)} = \push{\push{\sigma}{p}}{(T^i)} = \push{\mu}{(T^i)} = \mu
$$
so $\push{\sigma}{(S^i)} \in M$ for all $i$ and since $M$ is convex $\nu_n \in M$ for all $n$. So, by compactness of $M$ there exists a subsequence $\nu_{n_1}, \nu_{n_2}, \nu_{n_3}, \ldots$ convergent to some $\nu \in M$. Then $\nu$ is $S$-invariant by the proof of the Bogolyubov-Krylov theorem (see \cite[Theorem 1.1]{sinai}). Hence $\nu$ is as required.
\end{proof}

The averaging trick can be used to prove Theorem \ref{thm:main2} provided $\Gamma$ admits a \folner sequence, i.e. an increasing sequence of finite sets $F_n \subset \Gamma$ such that $\Gamma = \bigcup_{n\in\N} F_n$ and $\lim_{n\to\infty} \frac{|gF_n\symdiff F_n|}{|F_n|}=0$ for all $g \in \Gamma$. This is the case for instance for amenable groups and for abelian semigroups. However, there exist amenable semigroups admitting no \folner sequences, so we need a different method to prove Theorem \ref{thm:main2}.

\begin{proof}[Proof of Theorem \ref{thm:main2}]
Let $M\subset P(Y)$ be the set of all measures $\sigma$ satisfying $\push{\sigma}{p}=\mu$. By Lemma \ref{lem:1}, $M$ is a non-empty convex compact subset of $P(Y)$.

Note that the action $\Gamma \acts Y$ induces an action $\Gamma \acts P(Y)$ by push-forwards: $\gamma \sigma = \push{\sigma}{(S_\gamma)}$. Also, $\Gamma M \subset M$. Indeed, for any $\gamma \in \Gamma$ and $\sigma \in M$
$$\push{\gamma\sigma}{p} = \push{\push{\sigma}{(S_\gamma)}}{p} = \push{\sigma}{(p\circ S_\gamma)} = \push{\sigma}{(T_\gamma \circ p)} = \push{\push{\sigma}{p}}{(T_\gamma)} = \push{\mu}{(T_\gamma)} = \mu.$$
Hence by Day's fixed-point theorem \cite{day} there exists $\nu \in M$ with $\nu = \push{\nu}{(S_\gamma)}$ for all $\gamma \in \Gamma$. 
\end{proof}

\subsection*{Acknowledgements}
The author is grateful to Feliks Przytycki, Marcin Sabok and Gabor Elek for helpful comments. Also thanks are due to the anonymous referee for corrections and suggestions which improved the exposition of results.

This research was partially supported by the NCN (National Science Centre, Poland) through the grant Harmonia no. 2018/30/M/ST1/00668.

%%%%%%%%%%% To ease editing, use normal size for the references:

\end{document}